\documentclass[12pt]{amsart}
\usepackage{graphicx}
\usepackage{amsmath}
\usepackage{amssymb}
\usepackage{amsthm}
\usepackage{tikz}

\newtheorem{thm}{Theorem}
\newtheorem{prop}[thm]{Proposition}
\newtheorem{lem}[thm]{Lemma}
\newtheorem{cor}[thm]{Corollary}
\newtheorem{ques}{Question}

\newtheorem*{thm*}{Theorem}

\theoremstyle{definition}
\newtheorem{defn}[thm]{Definition}

\def\PA{\mathsf{PA}}

\def\lpa{{\mathcal L}_{\PA}}
\def\lpax{{\mathcal L}_{\PA} \cup \{ X \}}

\newcommand{\anglebracket}[1]{\langle #1 \rangle}

\newcommand{\mc}[1]{\mathcal{#1}}

\DeclareMathOperator{\Lt}{Lt}

\DeclareMathOperator{\scl}{Scl}
\DeclareMathOperator{\Def}{Def}

\DeclareMathOperator{\dcl}{dcl}

\DeclareMathOperator{\dom}{dom}

\begin{document}

\title{CP-generic expansions of models of Peano Arithmetic}
\author{Athar Abdul-Quader}
\address{School of Natural and Social Sciences, 
SUNY Purchase College,
735 Anderson Hill Road,
Purchase, NY 10577}
\email{athar.abdulquader@purchase.edu}
\author{James H. Schmerl}
\address{Department of Mathematics,
University of Connecticut,
Storrs, CT 06269}
\email{james.schmerl@uconn.edu}

\begin{abstract}
	We study notions of genericity in models of $\PA$, inspired by lines of inquiry initiated by Chatzidakis and Pillay and continued by Dolich, Miller and Steinhorn in general model-theoretic contexts. These papers studied the theories obtained by adding a ``random" predicate to a class of structures. Chatzidakis and Pillay axiomatized the theories obtained in this way. In this article, we look at the subsets of models of $\PA$ which satisfy the axiomatization given by Chatzidakis and Pillay; we refer to these subsets in models of $\PA$ as CP-generics. We study a more natural property, called strong CP-genericity, which implies CP-genericity. We use an arithmetic version of Cohen forcing to construct (strong) CP-generics with various properties, including ones in which every element of the model is definable in the expansion, and, on the other extreme, ones in which the definable closure relation is unchanged.
\end{abstract}

\keywords{models of arithmetic, expansions, genericity, definability}
\subjclass[2010]{Primary 03C62, 03H15}

\maketitle

\noindent
\section{Introduction}

In \cite{chatzidakis_pillay}, Chatzidakis and Pillay studied ``generic'' expansions of theories. Given a first order theory $T$ in a language $\mc{L}$, and a unary predicate $P \not \in \mathcal{L}$, Chatzidakis and Pillay axiomatized the theory $T_P$, the model companion of $T$ in the language $\mathcal{L} \cup \{ P \}$. This is, roughly, the theory of the expansion of models of $T$ by a generic unary predicate. Dolich, Miller and Steinhorn, in \cite{dms13} and \cite{dms16}, continued this work in studying the notion of adding ``generic'' predicates to o-minimal theories.

In the context of arithmetic, the methods used in the above papers do not work as well. However, one can take a particular model $\mc{M} \models \PA$ and consider the Chatzidakis-Pillay conditions on so-called ``generic" subsets of $M$. This line of inquiry was explored in \cite{neutral}, and we investigate this further here.

\subsection{Background}

All models in this article are models of {\sf PA} and their expansions. We use $\mathcal M$, $\mathcal N$, $\mathcal K$, etc.~for models of {\sf PA}, and $M$, $N$, $K$, etc.~for their respective domains.

By convention, definability refers to definability with parameters. If a set is definable without parameters, then we say that it is 0-definable. Given a model $\mc{M} \models \PA$ and $a \in M$, $\scl(a)$ denotes the Skolem closure of $a$; since $\PA$ has definable Skolem terms, this coincides with $\dcl(a)$. For models of $\PA$, we refer to Skolem closures rather than definable closures in the rest of this article for this reason; for expansions of models of $\PA$ which do not have definable Skolem terms, we continue to refer to definable closures.

The following two definitions appeared in \cite{neutral}.

\begin{defn} \label{cpg}A subset $X$ of $M$ is called {\it CP-generic} if whenever $D \subseteq M^n$ is definable using only $a \in M$ as a parameter and $I \subseteq \{ 1, \ldots, n \}$, if there are distinct $b_1, \ldots b_n \in M$ such that each $b_i \not \in \scl(a)$ and $\anglebracket{b_1, \ldots, b_n} \in D$, then there is $\anglebracket{b_1, \ldots, b_n} \in D$ such that $b_i \in X$ iff $i \in I$.
\end{defn}

\begin{defn}
	A subset $X$ of $M$ is  \emph{neutral} if for all $a$ in $M$, $\scl(a)$ coincides with $\dcl^{(\mathcal{M},X)}(a)$.
\end{defn}

We define another notion of genericity in this paper.

\begin{defn}\label{cpg-strong}A subset $X$ of $M$ is {\it strongly CP-generic} if whenever $D \subseteq M^n$ is definable and $I \subseteq \{ 1, \ldots, n\}$, if there is an infinite $B \subseteq D$ such that for all $\anglebracket{b_1, \ldots, b_n} \neq \anglebracket{c_1, \ldots, c_n} \in B$, $b_i \neq b_j$ for all $i \neq j$ and $\{ b_1, \ldots, b_n \} \cap \{ c_1, \ldots, c_n \} = \emptyset$, then there is $\anglebracket{b_1, \ldots, b_n} \in D$ such that $b_i \in X$ iff $i \in I$.
\end{defn}

Neutrality was explored in \cite{neutral}. In this article we explore the relationship between CP-genericity and neutrality, answering, in particular, Problem 1.3 from \cite{neutral} which asked if CP-generics are necessarily neutral. The relationship between CP-genericity and strong CP-genericity will be made clear in Section \ref{strong-cp}. Section \ref{main} contains our main results. First, we show that CP-generics exist for all countable models, and moreover, that every countable, recursively saturated model of $\PA$ has a CP-generic which is not neutral, answering Problem 1.3 in \cite{neutral} negatively. In fact, we show a complete failure of neutrality in some cases. We also show a construction of a neutral CP-generic. In Section \ref{cuts-and-classes}, we examine some properties of CP-generics in comparison to properties exhibited by neutral sets. We end, in Section \ref{opens}, with some open questions.

\section{Strong CP-genericity}\label{strong-cp}

In \cite{neutral}, the notion of CP-genericity was introduced, but not explored. In this article, we introduce \emph{strong CP-genericity}. The following results show that strong CP-genericity is equivalent to CP-genericity in the case of recursively saturated models.

\begin{prop}\label{equivalence} Let $\mc{M} \models \PA$ and $X \subseteq M$.
	\begin{enumerate}
		\item If $X$ is strongly CP-generic, then $X$ is CP-generic.
		\item If $\mc{M}$ is recursively saturated and $X$ is CP-generic, then $X$ is strongly CP-generic.
	\end{enumerate}
\end{prop}

In general, the converse to (1) is not true. For example, if $\mc{M}$ is prime, then every subset $X \subseteq M$ is CP-generic, although every strong CP-generic is infinite.

\begin{proof}
Suppose $X$ is strongly CP-generic. Let $D \subseteq M^n$ be definable using parameter $a \in M$, and let $I \subseteq \{ 1, \ldots, n \}$. Assume there exist distinct $b_1, \ldots, b_n$ such that $b_i \not \in \scl(a)$ for each $i$ and $\anglebracket{b_1, \ldots, b_n} \in D$. Then one can construct a $B$ as in Definition \ref{cpg-strong}: choose the smallest tuple $\bar{b}_0 = \anglebracket{b_{0, 1}, \ldots, b_{0, n}} \in D$ where the $b_{0, i}$ are pairwise distinct, and inductively, if $\bar{b}_0, \ldots, \bar{b}_{m-1}$ have been chosen, choose $\bar{b}_m = \anglebracket{b_{m, 1}, \ldots, b_{m, n}} \in D$ such that $b_{m, i}$ are all pairwise distinct and all differ from the previously chosen $b_{j, k}$ ($j < m, 1 \leq k \leq n$). For each $m \in \omega$ and $1 \leq k \leq n$, $b_{m, k} \in \scl(a)$, and so at each finite stage there is a tuple distinct from those previously chosen. Therefore $B = \{ \anglebracket{b_{m, 1}, \ldots, b_{m, n}} : m \in \omega \} \subseteq D$ satisfies the hypothesis for $B$ in Definition \ref{cpg-strong}. By strong CP-genericity, then, there is $\anglebracket{b_1, \ldots, b_n} \in D$ such that $b_i \in X$ if and only if $i \in I$.

Now assume $\mc{M}$ is recursively saturated and $X$ is CP-generic. Let $D \subseteq M^n$ be definable using parameter $c \in M$, and suppose it satisfies the hypothesis of Definition \ref{cpg-strong}. Let $I \subseteq \{ 1, \ldots, n \}$. Consider the type 
\begin{align*}
p(x_1, \ldots, x_n) = &\{ \anglebracket{x_1, \ldots, x_n} \in D \} \cup \{ x_i \neq x_j : 1 \leq i, j \leq n, i \neq j \} \cup \\ & \{ t_m(c) \neq x_i : 1 \leq i \leq n, m \in \omega \},
\end{align*}
where the $t_m$ range over all unary $\lpa$ Skolem terms (in some recursive enumeration of such Skolem terms). By assumption, there is an infinite $B \subseteq D$ such that for all distinct $\bar{b}, \bar{c} \in B$, the sets $\{ b_1, \ldots, b_n \}$ and $\{ c_1, \ldots, c_n \}$ are disjoint. This implies that $p(x_1, \ldots, x_n)$ is finitely satisfiable. If $\anglebracket{x_1, \ldots, x_n}$ realizes $p$, then the $x_i$ are pairwise distinct and $x_i \not \in \scl(c)$ for each $i$, satisfying the hypothesis for CP-genericity. Then since $X$ is CP-generic, there are $\anglebracket{x_1, \ldots, x_n} \in D$ such that $x_i \in X$ iff $i \in I$.
\end{proof}

\section{Main Results}\label{main}

The goal of this section is to answer Problem 1.3 from \cite{neutral}: if $X$ is CP-generic, must $X$ necessarily be neutral? We see in Theorem \ref{not-neutral} that the answer is no; moreover, in Corollaries \ref{neutral-fail-total} and \ref{every-subset}, that in fact, neutrality can fail badly in (strong) CP-generics. Lastly, we show in Corollary \ref{neutral-strong-cp} that every countable model of $\PA$ has neutral strong CP-generics.

Many of the proofs in this section use \emph{$\mc{M}$-Cohen generics}, which we define as follows.

\begin{defn}
	Let $\mc{M} \models \PA$. Consider the notion of forcing in $(\mc{M}, \omega)$ whose conditions are functions $p : A \to \{0, 1 \}$, where $A \subseteq M$ is finite. Define $p \trianglelefteq q$ if $q$ extends $p$. Let $G \subseteq M$ be generic for this notion of forcing. Then $X = \{ a \in M : p(a) = 0 \text{ for some } p \in G \}$ is an \emph{$\mc{M}$-Cohen generic}.
\end{defn}

For a review of the terminology of arithmetic forcing, see \cite[Chapter 6]{ks}. The usual forcing and truth lemmas hold; see \cite[Lemma 6.2.6]{ks}. That is: forcing is definable in $(\mc{M}, \omega)$, and for any formula $\theta(\overline{x})$ in the expanded language and any $\overline{m} \in M$, $(\mc{M}, \omega, X) \models \theta(\overline{m})$ if and only if there is $p \in G$ such that $(\mc{M}, \omega) \models (p \Vdash \theta(\overline{m}))$. In particular, $p \Vdash m \in X$ iff $m \in \dom(p)$ and $p(m) = 0$.

The reader should be cautioned here that for nonstandard $\mc{M}$, $\mc{M}$-Cohen generics are not the same as Cohen generics in the sense of \cite[Chapter 6]{ks}. That is, given a model $\mc{M}$, Cohen forcing (in the sense of \cite[Chapter 6]{ks}) is the notion of forcing whose conditions are definable functions $p : [0, m) \to \{ 0, 1 \}$, for $m \in M$. If $G$ is generic for this forcing, then the set $X = \{ a \in M : p(a) = 0 \text{ for some } p \in G \}$ is referred to as a \textit{Cohen generic}.

\begin{lem}\label{cohen-implies-strong}Every $\mc{M}$-Cohen generic is strongly CP-generic.\end{lem}

\begin{proof}
Let $p$ be a condition. Let $D \subseteq M^n$ be a definable set such that there is an infinite $B \subseteq D$ as in Definition \ref{cpg-strong}. We show that for each $I \subseteq \{ 1, \ldots n \}$, there is $q_I \trianglerighteq p$ such that \begin{equation}\label{forcing-cpg-strong}q_I \Vdash \exists x_1 \ldots \exists x_n [\anglebracket{x_1, \ldots, x_n} \in D \wedge \bigwedge\limits_{i \leq n} (x_i \in X \leftrightarrow i \in I)].\end{equation} Since $p : A \to \{ 0, 1 \}$ is finite, there are $x_1, \ldots, x_n \not \in A$ such that $\mc{M} \models \anglebracket{x_1, \ldots, x_n} \in D$. Take $A^\prime = A \cup \{ x_i : 1 \leq i \leq n \}$ and $q_I = p \cup \{ \anglebracket{x_i, 0} : i \in I \} \cup \{ \anglebracket{x_i, 1} : i \not \in I \}$. 

If $X$ and $G$ are as in the definition of $\mc{M}$-Cohen genericity, then for each $D \subseteq M^n$ definable satisfying the hypothesis in Definition \ref{cpg-strong}, and each $I \subseteq \{1, \ldots, n \}$, there is $q_I \in G$ satisfying \eqref{forcing-cpg-strong}. Therefore $X$ is strongly CP-generic.
\end{proof}

The following lemma can be proven using the standard proof that generics exist for countable models.
\begin{lem}\label{cohens-exist}If $\mc{M}$ is countable, then $\mc{M}$-Cohen generics exist.\qed\end{lem}

In fact, if $\mc{M}$ is uncountable, there are no $\mc{M}$-Cohen generics.

\begin{cor}\label{existence}For any countable $\mc{M}$, there is $X \subseteq M$ that is strongly CP-generic.\qed\end{cor}

\begin{lem}\label{strong-above-omega}If $\mc{M}$ is nonstandard and $X, Y \subseteq M$ are such that $X \setminus \omega = Y \setminus \omega$, then $X$ is strongly CP-generic iff $Y$ is.
\end{lem}
\begin{proof}
Suppose $X$ is strongly CP-generic. Let $D \subseteq M^n$ and $I \subseteq \{1, \ldots, n\}$ be such that there is an infinite $B \subseteq D$ as in Definition \ref{cpg-strong}. By overspill, there is $c > \omega$ such that \[D^\prime = \{ \anglebracket{x_1, \ldots, x_n} : \mc{M} \models \anglebracket{x_1, \ldots, x_n} \in D \wedge \bigwedge\limits_{1 \leq i \leq n} x_i > c \}\] also has such an infinite subset $B$. By strong CP-genericity applied to $D^\prime$, there is $\anglebracket{b_1, \ldots, b_n} \in D^\prime$ such that $b_i \in X$ iff $i \in I$. Since each $b_i > \omega$, it follows that $b_i \in Y$ iff $i \in I$.
\end{proof}

Note that in the standard model $\mathbb{N}$, the $\mathbb{N}$-Cohen generics are exactly the Cohen generics in the sense of \cite[Chapter 6]{ks}. By Lemma \ref{cohen-implies-strong}, every Cohen generic in the standard model is strongly CP-generic. The converse is false: there are strong CP-generics which are not Cohen generic. Moreover, for every countable $\mc{M}$, there are strong CP-generics which are not $\mc{M}$-Cohen generic.

\begin{prop}For every countable $\mc{M}$, there is a strong CP-generic $X \subseteq \omega$ which is not $\mc{M}$-Cohen generic.
\end{prop}

\begin{proof}
	If $\mc{M}$ is standard and $X$ is Cohen generic, then there is $n \in \omega$ such that $[n, 2n] \subseteq X$. One confirms this by noticing that the set of conditions which force this is dense. However, one can routinely construct a strong CP-generic which avoids $[n, 2n]$ as a subset for each $n$.
	
	If $\mc{M}$ is nonstandard, then by Lemma \ref{strong-above-omega}, if $X$ is strongly CP-generic, then $X \setminus \omega$ is also strongly CP-generic. However, every $\mc{M}$-Cohen generic must nontrivially intersect $\omega$.
\end{proof}

Next we answer Problem 1.3 from \cite{neutral} in the negative.

\begin{thm}\label{not-neutral}Let $\mc{M}$ be any countable, non-prime model. Then there is $X \subseteq M$ such that $X$ is (strongly) CP-generic but not neutral.
\end{thm}

\begin{proof}
We proceed by first constructing an $\mc{M}$-Cohen generic with the property that for all $a \neq b$, there is $n \in \omega$ such that $a + n \in X$ iff $b + n \not \in X$. Then, by Lemma \ref{strong-above-omega}, given any $a \not \in \scl(0)$, the set $X^\prime = (X \setminus \omega) \cup \{ 2n + 1 : n \in \omega, a + n \in X \}$ is also strongly CP-generic, and we will see that $a$ is definable in $(\mc{M}, X^\prime)$.

To construct the $\mc{M}$-Cohen generic with the requisite property, enumerate the model as $(a_i : i \in \omega)$ and the dense definable sets in the $\mc{M}$-Cohen forcing in $(\mc{M}, \omega)$ as $(A_i : i \in \omega)$. At stage $2i$, meet $A_i$ with condition $p_{2i}$ (extending $p_{2i-1}$ if $i > 0$). At stage $2i + 1$, let $n \in \omega$ be the least such that $a_j + n \not \in \dom(p_{2i})$ for each $j < i$, and extend $p_{2i}$ to $p_{2i+1}$ such that $p_{2i+1}(a_i + n) = 0$ and $p_{2i+1}(a_j + n) = 1$ for each $j < i$.

Let $X$ be the resulting $\mc{M}$-Cohen generic. Then for all $a \neq b$ nonstandard, there is $n \in \omega$ such that $a + n \in X \iff b + n \not \in X$.

To complete the proof, let $a \not \in \scl(0)$ and let $X^\prime = (X \setminus \omega) \cup \{ 2n + 1 : a + n \in X \}$. By Lemma \ref{strong-above-omega}, $X^\prime$ is strongly CP-generic. Moreover, notice that $\omega$ is $0$-definable in $(\mc{M}, X)$ as 
\begin{equation}\label{omega}
\{ n : (\mc{M}, X) \models \forall i < n (2i \not \in X) \}.
\end{equation} To see this, let $J$ be the set of those $n$ satisfying \eqref{omega}. By definition, $\omega \subseteq J$. If $c > \omega$, consider the set $D = \{ \anglebracket{x_0, x_1} : \mc{M} \models x_0 < x_1 < c \wedge x_1 = 2 \cdot x_0 \}$. Since $c$ is nonstandard, $D$ contains an infinite set $B$ such that if $\anglebracket{b_0, b_1} \neq \anglebracket{c_0, c_1} \in B$, then $b_0, b_1, c_0, c_1$ are all distinct. By strong CP-genericity, there is $\anglebracket{x_0, x_1} \in D$ such that $x_1 \in X$. Then, since $x_1$ is even, $c \not \in J$.

Lastly, $a$ is definable in $(\mc{M}, X^\prime)$ as $x = a$ iff $$(\mc{M}, X^\prime) \models \forall n \in \omega (x + n \in X \leftrightarrow 2n + 1 \in X).\qedhere$$
\end{proof}

We can further modify the above idea to show that there is a strong CP-generic $X \subseteq M$ such that every element of $M$ is definable in $(\mc{M}, X)$. Instead of using the evens and odds, as we did above, take a partition of $\omega$ into countably many uniformly definable disjoint infinite sets $(I_j : j \in \omega)$. For example, let $I_j$ be the powers of the $j$-th prime. Then ensure $I_0 \cap X = \emptyset$, and for each $a_j \in M$, put the $n$-th element of $I_{j+1}$ in $X$ if and only if $a_j + n \in X$. In this way, we obtain the following corollary:

\begin{cor}\label{neutral-fail-total}Let $\mc{M}$ be countable. Then there is a strong CP-generic $X \subseteq M$ such that every element of $M$ is definable in $(\mc{M}, X)$.\qed
\end{cor}

Note that if $\mc{M}$ is prime, every element of $M$ is already definable, regardless of what $X \subseteq M$ is taken.

By another similar modification to the proof of Theorem \ref{not-neutral}, we find that for every countable $\mc{M}$ and $A \subseteq M$, there is a strong CP-generic $X \subseteq M$ such that $A \in \Def(\mc{M}, X)$. Moreover, there is a kind of uniformity in defining these functions.

\begin{cor}\label{every-subset}There is a formula $\phi(x) \in \lpax$ such that for any countable $\mc{M}$ and any subset $A \subseteq M$, there is a strong CP-generic $X$ such that $\phi$ defines $A$ in $(\mc{M}, X)$.
\end{cor}
	
\begin{proof}
	Let $\mc{M}$ and $A$ be given. The formula $\phi(x)$ (and the uniformity of it) will be clear as part of the construction of $X$.
	
	In the standard model, one can routinely construct a Cohen generic $X$ such that $A \in \Def(\mathbb{N}, X)$, similar to the proof of \cite[Theorem 6.2.11]{ks}. Using Lemma \ref{cohen-implies-strong}, $X$ is also strongly CP-generic.
	
	If $\mc{M}$ is nonstandard, we modify the construction in Theorem \ref{not-neutral}. Start by constructing a strong CP-generic $X$ with the property that for all $a \neq b$, there is $n \in \omega$ such that $a + n \in X$ iff $b + n \not \in X$. Note that the following construction can be done for any nonstandard $\mc{M}$, while Theorem \ref{not-neutral} only applies to non-prime models.
	
	Fix an enumeration of $A$ in order type $\omega$ as $(a_n : n \in \omega)$. Partition $\omega$ into definable, disjoint, infinite sets $I_0, I_1, I_2, \ldots$.  Define $X^\prime$ as follows: for $c > \omega$, $c \in X$ iff $c \in X^\prime$, so that $X^\prime$ is strongly CP-generic by Lemma \ref{strong-above-omega}. For $n \in I_0$, ensure $n \not \in X^\prime$ so that $\omega$ is definable (as above). For $n \in A \cap \omega$, put the $n$-th element of $I_1$ in $X^\prime$. Then for $x \in A \setminus \omega$, if $x = a_n$, put the $m$-th element of $I_{n+2}$ in $X^\prime$ if and only if $x + m \in X^\prime$.
	
	Now $A$ is definable in $(\mc{M}, X^\prime)$ as $x \in A$ if $x \in \omega$ and the $x$-th element of $I_1$ is in $X^\prime$, or $x \not \in \omega$ and there is $n \in \omega$ such that for all $m \in \omega$, $x + m \in X^\prime$ if and only if the $m$-th element of $I_{n+2}$ is in $X^\prime$.
	
	One checks that the definition of $A$ can be made uniform by noticing that there is a statement true in $(\mc{M}, X)$ (for all countable$\mc{M}$ and strong CP-generics constructed above) iff $\mc{M}$ is standard.
\end{proof}	

\begin{cor}Every countable $\mc{M}$ has $2^{\aleph_0}$ distinct strongly CP-generic subsets.\qed\end{cor}
	
Conversely, there are neutral sets which are CP-generic. We again use $\mc{M}$-Cohen generics to establish this. As seen in Lemma \ref{cohen-implies-strong}, $\mc{M}$-Cohen generics are strongly CP-generic. Here we see that they are also neutral.

\begin{thm}\label{cohen-is-neutral}For any countable $\mc{M}$, every $\mc{M}$-Cohen generic is neutral.\end{thm}
\begin{proof}
Let $X$ be an $\mc{M}$-Cohen generic and let $G$ be a corresponding generic set of conditions, so that $X = \{ m : p(m) = 0 \text{ for some } p \in G\}$. By Lemma \ref{cohen-implies-strong}, it is strongly CP-generic.
 
We show that $\dcl^{(\mc{M}, \omega, X)} = \dcl^{(\mc{M}, \omega)}$. By Kanovei \cite{kan}, generalized in \cite[Theorem 8.4.7]{ks}, the $\dcl$ relation in $(\mc{M}, \omega)$ is identical to that in $\mc{M}$, which shows that $X$ is neutral. Before we show this, we first show a lemma about compatibility of conditions, which will be critical for various stages of our proof.

\begin{lem}\label{infcompat}
Let $\{ X_i : i \in \omega \}$ be a family of infinite sets of conditions. If for all $i, j \in \omega$ and $p \in X_i, q \in X_j$, $|p| = |q|$, then there are $i \neq j \in \omega$, with $p \in X_i$, $q \in X_j$, and $p \neq q$ such that $p$ and $q$ are compatible.
\end{lem}

Note that the $X_i$ need not be pairwise distinct.

\begin{proof}
Let $n$ be the cardinality of the domain of any condition in (any of the) $X_i$. We prove this by induction on $n$.

If $n = 0$, there are no such $X_i$, since there is only one condition of whose domain is empty (the empty condition). If $n = 1$, the Lemma holds by pigeonholing.

Inductively suppose the Lemma holds for all collections $\{ Y_j : j \in \omega \}$ such that $|p| < n$ for each $p$ in (any of the) $Y_j$. Fix $p \in X_0$. If $p$ is not compatible with any $q \neq p$ for all $q \in \bigcup\limits_{i \geq 1} X_i$, then there is $a \in \dom(p)$ such that there are infinitely many $j \in \omega$, and infinitely many $q \in X_j$ with $a \in \dom(q)$ but $p(a) \neq q(a)$. Without loss of generality, assume $p(a) = 0$, so for all such $q$, $q(a) = 1$.

For such a $q$, define $q^*$ as $q \setminus \{ \anglebracket{a, 1}\}$. Let $Y_j$ be the $j$-th set in the collection $\{ X_i : i \in \omega \}$ such that there are infinitely many $q$ with $q(a) = 1$. Let $Y_j^* = \{ q^* : q \in Y_j, q(a) = 1 \}$. Then $\{ Y_j^* : j \in \omega \}$ satisfies the inductive hypothesis, and so there are $j_0$ and $j_1$, and $q_0^* \in Y_{j_0}^*$, $q_1 \in Y^*_{j_1}$ such that $q_0^*$ and $q_1^*$ are compatible. Then $q_0 = q_0^* \cup \anglebracket{a, 1}$ and $q_1 = q_1^* \cup \anglebracket{a, 1}$ are also compatible.
\end{proof}

Now we return to the proof that $\dcl^{(\mc{M}, \omega, X)} = \dcl^{(\mc{M}, \omega)}$. Suppose $a, b \in M$ are such that $(\mc{M}, \omega, X) \models \forall x (\phi(x, b) \leftrightarrow x = a)$ for some $\phi$ in the expanded language. There is $p \in G$ such that \begin{equation}\label{forcdcl}(\mc{M}, \omega) \models p \Vdash [\forall x (\phi(x, b) \leftrightarrow x = a)].\end{equation} Let $p$ be such that it satisfies \eqref{forcdcl} and $|p|$ is minimal. Let $Y = \{ q : |q| = |p| \text{ and } q \Vdash [\forall x (\phi(x, b) \leftrightarrow x = a)] \}$. We consider the two cases of whether $Y$ is finite or $Y$ is infinite.

If $Y$ is finite, then $p \in \dcl^{(\mc{M}, \omega)}(a, b)$. Since $\omega$ is neutral, then $p \in \scl(a, b)$, and so there are $n \in \omega,$ Skolem terms $t_0, \ldots, t_{n-1}$, and $\sigma : [0, n-1] \to \{ 0, 1 \}$ such that $p(t_i(a, b)) = \sigma(i)$ for each $i < n$. Let $p(x)$ be the finite function defined by $t_i(x, b) \mapsto \sigma(i)$ for $0 \leq i < n$, so that $p = p(a)$. Now consider the set $B = \{ c : (\mc{M}, \omega) \models p(c) \Vdash \forall x [ (\phi(x, b) \leftrightarrow x = c)]\}$. Clearly $a \in B$, and so if $B$ is finite then $a \in \dcl^{(\mc{M}, \omega)}(b)$. If $B$ is infinite, let $Z = \{ p(c) : c \in B \}$  and apply Lemma \ref{infcompat} to the collection $\{ X_i : i \in \omega \}$ where each $X_i = Z$. We obtain compatible conditions $p(c_1) \neq p(c_2) \in Z$. But in $(\mc{M}, \omega)$, \[p(c_1) \cup p(c_2) \Vdash \forall x [\phi(x, b) \leftrightarrow x = c_1] \wedge \forall x [\phi(x, b) \leftrightarrow x = c_2], \] which is impossible.

If $Y$ is infinite, for each $c \in M$ let $X_c = \{ q : |q| = |p| \text{ and } q \Vdash [\forall x (\phi(x, b) \leftrightarrow x = c)]\}$. If there are only finitely many $c$ such that $X_c$ is infinite, then $a \in \dcl^{(\mc{M}, \omega)}(b)$, so assume that there are infinitely many such $c$. Applying Lemma \ref{infcompat}, there are $c_1 \neq c_2$ with $p \in X_{c_1}, q \in X_{c_2}$ and $p$ and $q$ are compatible. But then $p \cup q \Vdash [\forall x(\phi(x, b) \leftrightarrow x = c_1)] \wedge [\forall x (\phi(x, b) \leftrightarrow x = c_2)]$, which is impossible.
\end{proof}

Combining Lemma \ref{cohen-implies-strong}, Lemma \ref{cohens-exist}, and Theorem \ref{cohen-is-neutral}, we obtain the following:

\begin{cor}\label{neutral-strong-cp}Every countable $\mc{M}$ has a neutral, strong CP-generic.\qed\end{cor}

\section{Cuts and Classes}\label{cuts-and-classes}

In every model $\mc{M}$, the standard cut $\omega$ is neutral. This is not true in general for CP-generics, per the following result. In the following, an extension $\mc{N} \prec \mc{M}$ is called \emph{superminimal} (see \cite[Section 2.1.2]{ks}) if whenever $b \in M \setminus N$, then $\scl(b) = M$.

\begin{prop}\label{unbd}For any $\mc{M}$, the following are equivalent:
	
	\begin{enumerate}
		\item\label{bounded-cpg} $\mc{M}$ has a bounded CP-generic subset $X \subseteq M$.
		\item\label{suff-large-gen} All sufficiently large $b \in M$ generate $\mc{M}$.
		\item\label{supermin-end} $\mc{M}$ is prime or is a superminimal elementary end extension of some $\mc{N} \prec \mc{M}$.
	\end{enumerate}
\end{prop}

One notes here that conditions \eqref{suff-large-gen} and \eqref{supermin-end} both imply that $\mc{M}$ is countable.

\begin{proof}
	The implication $\eqref{supermin-end} \implies \eqref{suff-large-gen}$ is clear from definitions. To show $\eqref{suff-large-gen} \implies \eqref{supermin-end}$, suppose \eqref{suff-large-gen} holds. Let $b \in M$ be such that $\scl(b) = \mc{M}$. Let $K$ be the set of those $a \in M$ such that $b \not \in \scl(a)$. Then by \eqref{suff-large-gen}, either $K = \emptyset$ or $K$ is a proper cut of $\mc{M}$. If $K = \emptyset$, then $\mc{M} = \scl(0)$. If $K \neq \emptyset$, then we claim that $\mc{K}$ is an elementary submodel of $\mc{M}$, and $\mc{M}$ is a superminimal elementary end extension of $\mc{K}$. To see this, suppose $\mc{M} \models \exists x (\phi(x, a))$, for some $a \in K$. Then there is $c \in \scl(a)$ such that $\mc{M} \models \phi(c, a)$. Then since $\scl(c) \subseteq \scl(a)$ and $b \not \in \scl(a)$, it follows that $b \not \in \scl(c)$ and so $c \in K$. Moreover, for any $a \in M \setminus K$, $\scl(a) = \mc{M}$, and so $\mc{M}$ is a superminimal extension of $\mc{K}$.
	
	Next we show $\eqref{bounded-cpg} \implies \eqref{suff-large-gen}$. Suppose $X$ is bounded and is CP-generic. Let $b > X$. If $\scl(b) \neq \mc{M}$, then there is $x \not \in \scl(b)$ such that $x > b$. Then by CP-genericity,
	\[(\mathcal{M}, X) \models \exists x (x > b \wedge x \in X),\]
	immediately contradicting the assertion that $X$ is bounded above by $b$.
	
	Finally we show $\eqref{supermin-end} \implies \eqref{bounded-cpg}$. If $\mc{M}$ is prime, then every $X \subseteq M$ is CP-generic by definition. If $\mc{N} \prec \mc{M}$, and $\mc{M}$ is a superminimal elementary end extension of $\mc{N}$, then we build $X \subseteq N$ which is CP-generic in $\mc{M}$. To build $X$, we construct finite sets $A_i, B_i$, for each $i \in \omega$, such that the following hold:
	
	\renewcommand{\theenumi}{\alph{enumi}}
	\begin{enumerate}
		\item\label{increasing} if $i < j \in \omega$, then $A_i \subseteq A_j$,
		\item\label{pos-and-neg} for all $i \in \omega$, $A_i \cap B_i = \emptyset$, 
		\item\label{implies-bounded} for all $i \in \omega$, $A_i \cup B_i \subseteq N$, and
		\item\label{implies-cpg} for every definable set $D \subseteq M^n$, $I \subseteq \{1, \ldots, n \}$, if $D$ satisfies the hypothesis of Definition \ref{cpg}, then there is $i < \omega$ and $\anglebracket{x_1, \ldots, x_n} \in D \cap (A_i \cup B_i)$ such that for $1 \leq j \leq n$, $x_j \in A_i$ iff $j \in I$ and $x_j \in B_i$ iff $j \not \in I$.
	\end{enumerate}
To begin the construction, let $A_0 = B_0 = \emptyset$. At stage $i+1$, suppose $A_i$ and $B_i$ are defined and that we are considering a definable $D \subseteq M^n$ and $I \subseteq \{1, \ldots, n \}$ satisfying the hypothesis of Definition \ref{cpg}. By superminimality, if $D$ is definable from $a \in M \setminus N$, there is nothing to show, so assume $D$ is definable from parameter $a \in N$, and there are $b_1, \ldots, b_n \not \in \scl(a)$ such that $\anglebracket{b_1, \ldots, b_n} \in D$. Since $A_i$ and $B_i$ are finite, then let $\anglebracket{x_1, \ldots, x_n} \in D$ be the smallest such that $x_j \not \in A_i\cup B_i$ for any $1 \leq j \leq n$. Then notice that each $x_j \in \mc{N}$, as the tuple $\anglebracket{x_1, \ldots, x_n} $ is definable from elements of $\mc{N}$. Let $A_{i+1} = A_i \cup \{ x_j : j \in I \}$ and $B_{i+1} = B_i \cup \{ x_j : j \not \in I \}$.

As there are only countably many pairs $(D, I)$ where $D$ is definable, $D \subseteq M^n$ for some $n \in \omega$ and $I \subseteq \{ 1, \ldots, n\}$, each such $D$ and $I$ is handled at some stage $i \in \omega$, so \eqref{implies-cpg} holds. Let $X = \bigcup\limits_{i \in \omega} A_i$. Properties \eqref{increasing}, \eqref{pos-and-neg}, and \eqref{implies-cpg} imply that $X$ is CP-generic, and property \eqref{implies-bounded} implies that $X \subseteq N$ (and therefore is bounded).
\end{proof}

%\begin{prop}\label{unbd}Let $\mathcal{M}$ be a model that is not finitely generated. If $X \subseteq M$ is CP-generic, then $X$ is unbounded.\end{prop}

%\begin{proof}
%	Suppose $X$ is bounded and is CP-generic. Let $b > X$. Since $\scl(b) \neq \mc{M}$, there is $x \not \in \scl(b)$ such that $x > b$. Then by CP-genericity,
%	\[(\mathcal{M}, X) \models \exists x (x > b \wedge x \in X),\]
%	immediately contradicting the assertion that $X$ is bounded above by $b$.
%\end{proof}

This shows that, in recursively saturated models (or any model which is not finitely generated), no proper cut is CP-generic . Of course, the word ``proper" can be omitted from the previous sentence, as one can verify that for any model $\mc{M}$, $M$ itself is not CP-generic.

\cite{neutral} focused on neutral \emph{classes} and neutral inductive sets. A subset $X$ of a model ${\mathcal M}$ is a \emph{class} if for each $a\in M$, $\{x \in X : {\mathcal M}\models x<a\}$ is definable in ${\mathcal M}$. A subset $X$ of $M$ is \emph{inductive} if $({\mathcal M},X)\models {\sf PA}^*$, i.e. the induction schema holds in $({\mathcal M},X)$ for all formulas of  the language of {\sf PA} with a unary predicate symbol interpreted as $X$. All inductive sets are classes. \cite[Corollary 3.3]{neutral} states that no undefinable neutral set in a recursively saturated model is a class. It turns out that, in recursively saturated models, no CP-generic is a class either. In fact, we have more: no strong CP-generic is a class in a nonstandard model.

\begin{thm}\label{not-class}Let $\mc{M}$ be nonstandard. If $X \subseteq M$ is a class, then it is not strongly CP-generic.
\end{thm}

\begin{proof}
	Let $X$ be strongly CP-generic and $b > \omega$. Suppose $B = [0, b) \cap X$ is infinite. If not, replace $X$ with its complement, which is also strongly CP-generic by definition. If $X$ is a class, then $B$ is an infinite definable subset of $M$, and so by strong CP-genericity, there is $x \in B$ such that $x \not \in X$.
\end{proof}

\section{Open Questions}\label{opens}

We close with some questions about CP-genericity. For the first question, we recall the notion of the \textit{substructure lattice} of a model. Given a structure $\mc{M}$, $\Lt(\mc{M}) = \{ \mc{K} : \mc{K} \prec \mc{M} \}$; see \cite[Chapter 4]{ks} for basic definitions and results on substructure lattices.

\begin{ques}Let $\mc{M}$ be a countable, recursively saturated model of $\PA$. For which subsemilattices $L$ of $\Lt(\mc{M})$ is it the case that there is a (strong) CP-generic $X$ such that the elementary substructures of $(\mc{M}, X)$ are exactly expansions of the $\mc{N} \prec \mc{M}$ such that $\mc{N} \in L$?
\end{ques}

Theorem \ref{not-class} asserts that strong CP-generic subsets of nonstandard models are not classes. In particular, this means that no CP-generic subset of a recursively saturated model is a class. In a similar vein to the ideas studied about neutrality in \cite{neutral}, we ask here if sets which are CP-generic (but not strongly CP-generic) can ever be classes in a nonstandard model. 

\begin{ques}For which $\mc{M}$ is there $X \subseteq M$ such that $X$ is a CP-generic class?\end{ques}

\begin{ques}
Let $\mc{M} \models \PA$ and $X \subseteq M$ strongly CP-generic. Is there always a proper $(\mc{N}, Y) \succ (\mc{M}, X)$ such that $Y$ is strongly CP-generic? Under what conditions on $(\mc{M}, X)$ can we always find such an $(\mc{N}, Y)$?
\end{ques}

\bibliographystyle{plain}
\bibliography{refs}

\end{document}